\documentclass[oneside,english]{amsart}
\usepackage[foot]{amsaddr}

\usepackage[T1]{fontenc}
\usepackage[latin9]{luainputenc}
\usepackage{amsmath}
\usepackage{graphicx}
\usepackage{mathtools}
\usepackage{amssymb}
\usepackage{amsthm}
\usepackage{tikz-cd}
\usepackage{mathrsfs}
\usepackage{pdfpages}
\usepackage{epstopdf}
\usepackage[colorinlistoftodos]{todonotes}
\usepackage{enumitem}
\usepackage{yfonts}
\usepackage{xcolor}
\usepackage{mathtools}
\usepackage{hyperref}
\hypersetup{colorlinks,citecolor=blue}

\makeatother

\usepackage{babel}

\usepackage{combelow}
\usepackage[bottom]{footmisc}

\newtheorem{abc}{Theorem}[section]

\newtheorem{abcd}{Corollary}[section]

\newtheorem{theorem}{Theorem}[section]
\newtheorem{lemma}[theorem]{Lemma}
\newtheorem{definition}[theorem]{Definition}
\newtheorem{example}[theorem]{Example}

\newtheorem{proposition}[theorem]{Proposition}
\newtheorem{corollary}[theorem]{Corollary}

\newtheorem*{theorem*}{Theorem}

\theoremstyle{remark}

\newcommand{\Z}{\mathbb{Z}}

\newcommand{\Prob}{\operatorname{Prob}}
\newcommand{\supp}{\operatorname{supp}}
\newcommand{\half}{\frac{1}{2}}

\makeatletter
\renewcommand{\email}[2][]{%
  \ifx\emails\@empty\relax\else{\g@addto@macro\emails{,\space}}\fi%
  \@ifnotempty{#1}{\g@addto@macro\emails{\textrm{(#1)}\space}}%
  \g@addto@macro\emails{#2}%
}
\makeatother
\title{Balanced Measures on Compact Median Algebras}

\author{Uri Bader}

\author{Aviv Taller}
\address{Email addresses:}
\address{A. Taller:
avivtaller@gmail.com}
\address{U. Bader: uribader@gmail.com}
\address{Data availability statement : NA}
\usepackage{lipsum}

\begin{document}
\maketitle

\begin{abstract} 
We initiate a systematic investigation of group actions on compact median algebras
via the corresponding dynamics on their spaces of measures.
We show that a probability measure which is invariant under a natural push forward operation must be a uniform measure on a cube
and use this to show that every amenable group action on a locally convex compact median algebra fixes a sub-cube.
\end{abstract}

\section{Introduction}
{\em Median algebras} form a common generalization of dendrites and distributive lattices.
While early investigations of these objects mainly dealt with combinatorial aspects as part of order theory,
recently they obtained much attention due to the observation that CAT(0) {\em cube complexes} carry a natural {\em median space} structure,
that is, compatible metric and median algebra structures.
A powerful tool in the investigation of a CAT(0) cube complex is provided by embedding it in a compact median algebra, namely its {\em Roller compactification}, see \cite{rol98} and also \cite{fio20} for a further discussion.
This brings to front the class of {\em compact median algebras}.

As common for dynamical systems, we study compact median algebras by investigating their invariant measures.
Let $M$ be a second countable compact median algebra, endowed with a continuous {\em median operator} $m:M^3\to M$.
We study the operator $\Phi$, called {\em self-median operator}, defined on the space of Borel probability measures of $M$
by the formula
\[ \Phi:\Prob(M) \to \Prob(M), \quad \Phi(\mu)=m_* (\mu^3), \]
along with its space of invariant measures, $\Prob(M)^\Phi$,
which elements we denote {\em balanced measures}.

A basic example of a second countable compact median algebra is the cube $\{0,1\}^I$, endowed with the product topology and median structure,
where $I$ is a countable index set\footnote{In this paper, a countable set could be finite, in particular empty.}.
One verifies easily that the uniform measure $\lambda \in \Prob(\{0,1\}^I)$ is balanced (see the beginning of \S\ref{sec:uniform} for details).
Clearly, if $f:N\to M$ is a continuous morphism of second countable compact median algebras and $\mu$ is a balanced measure on $N$ then $f_*(\mu)$ is a balanced measure on $M$.
In particular, for every continuous median algebra morphism $f:\{0,1\}^I \to M$,
the push forward of the uniform measure, $f_*(\lambda)$, is a balanced measure on $M$.
In case $f$ is injective, we say that $f(\{0,1\}^I)$ is a {\em cube} in $M$ and $f_*(\lambda)$ is a {\em cubical measure} on $M$.

Our main theorem is the following classification result
which applies to the class of {\em second countable, locally open-convex, compact}, or for short {\em sclocc}, median 
algebras (see Definitions~\ref{def:locc} and \ref{def:sclocc}).

\begin{abc} \label{thm:main}
Every balanced measure on a sclocc median algebra is cubical.
%
\end{abc}

Theorem~\ref{thm:main} is a direct corollary of the following two propositions.

\begin{proposition} \label{prop:support}
The support of every balanced measure on a sclocc median algebra is a cube.
\end{proposition}

\begin{proposition} \label{prop:uniform}
The uniform measure is the unique fully supported balanced measure on a cube.
\end{proposition}

Proposition~\ref{prop:support} has the following interesting corollary. 

\setcounter{abcd}{1}
\begin{abcd}
Assume that $G$ is a locally compact amenable group which acts on a sclocc median algebra $M$
by continuous automorphisms. 
Then $M$ contains a $G$-invariant cube.
\end{abcd}
 
Indeed, by the amenability of $G$, $\Prob(M)^G$ is non-empty
and clearly $\Phi$-invariant, thus it contains a balanced measure $\mu$ by the Tychonoff Fixed Point Theorem
and the support of $\mu$ is a $G$-invariant cube by Proposition~\ref{prop:support}.

Note that one cannot expect a $G$-invariant cube to be pointwise fixed.
Indeed, the group $(\Z/2\Z)^I$ acts transitively
on the cube $\{0,1\}^I$ by median automorphisms
(see the discussion in the begining of \S\ref{sec:uniform}).

The requirements that a topological median algebra is sclocc
is satisfied by the Roller compactification of a (possibly infinite dimensional) CAT(0) cube complex with countably many vertices, as well as a second countable median space of finite rank, see Example~\ref{ex:sclocc}. For further discussions on CAT(0) cube complexes and median spaces, see \cite{sag95, cfi16, bow13, Bowditch2020, cdh10, fio20} and the references therein.

We note that if $M$ is not assumed second countable then the median operator $m$ might not be Borel (with respect to the product of the Borel $\sigma$-algebras on $M^3$) even though it is continuous, thus we cannot define the self median operator $\Phi$ in this generality.
Consider the interval $[0,\omega_1]$,
where $\omega_1$ is the first uncountable ordinal,
and endow it with the order topology and median structure.
Even though this space is not second countable, one verifies easily that $m$ is Borel, thus $\Phi$ is defined, in this case. 
Then the Dieudonn\'e measure,
which assigns 1 to Borel sets containing an unbounded closed subset and 0 to their complements, is a balanced measure which has an empty support.
Hereafter, when discussing measures on compact spaces we will assume that the underlying space is second countable, thus every measure has a non-empty support.

The structure of this note is as follows.
After reviewing median algebras in the next section, we will focus our attention on the theory of compact median algebras in \S\ref{sec:compact}.
We will then prove Proposition~\ref{prop:support} in \S\ref{sec:support}
and Proposition~\ref{prop:uniform} in \S\ref{sec:uniform}.

\medskip
\noindent \textbf{Acknowledgments}. 
The authors thank Elia Fioravanti for sharing with them the observation that a finite rank median space is locally open-convex and referring them to \cite{Bowditch2020}.
The authors wish to thank the anonymous referee for many valuable comments.
This research is supported by ISF Moked 713510 grant number 2919/19.

\section{A review of Median Algebras} \label{sec:review}

A \textit{median algebra} is a set $M$, equipped with a ternary operation\\ $m:M\times M \times M \rightarrow M$, called the \textit{median operator}, that has the following properties: \\$\forall x,y,z,u,v\in M$
\begin{enumerate}[leftmargin=3cm]
    \item[(\textbf{Med 1})] $m(x,y,z)=m(x,z,y)=m(y,x,z)$       
    \item[(\textbf{Med 2})] $m(x,x,y)=x$     
    \item[(\textbf{Med 3})] $m(m(x,y,z),u,v)=m(x,m(y,u,v),m(z,u,v))$ \quad  
\end{enumerate}
A {\em median morphism} $\varphi:M\rightarrow N$ between two median algebras, is a map such that $\varphi\circ m = m\circ( \varphi \times \varphi \times \varphi)$ (notice that we may confuse often between the median operators of one object to another, like we did here). 

A basic example of a median algebra is $\{0,1\}$, endowed with the standard median operation.
The category of median algebra has arbitrary products, given by the Cartesian products of the underlying sets
and the coordinate-wise operations. 
In particular, for any index set $I$, $\{0,1\}^I$ is a median algebra.
A median algebra that isomorphic to $\{0,1\}^I$ for some $I$ is called a {\em cube}.
A cube in a median algebra $M$ means a median subalgebra of $M$ that is a cube.

For the rest of this section we fix a median algebra $M$.
The \textit{interval} of two elements  $x,y\in M$ is the subset $[x,y]=\{u\in M\mid m(x,y,u)=u\}$. 

\begin{lemma}[statements (Int 3) and (Int 6) at the beginning of section 2 in \cite{rol98}] \label{lem:int6}
For every $x,y,z\in M$,

\begin{enumerate}
\item $y\in [x,z] \quad \implies \quad [x,y]\subset [x,z]. $
\item $ y\in [x,z] \quad \iff \quad [x,y]\cap[y,z]=\{y\}. $
\end{enumerate}
\end{lemma}

A subset $C\subset M$ is called \textit{convex}, if for every $x,y\in C$, $[x,y]\subset C$. An important property of convex sets is what usually known as \textit{Helly's Theorem }.

\begin{lemma}[{\cite[Theorem 2.2]{rol98}}] \label{lem:helly}
 If $C_1,...,C_n$ are convex, and for every $i\neq j\ \ C_i\cap C_j\neq \varnothing$, then also $\overset{n}{\underset{i=1}{\cap}}C_i \neq \varnothing$.
\end{lemma}

Let $C\subset M$ be a convex set and $x\in M$. We say that $y\in C$ is the \textit{gate} for $x$ in $C$, if $y\in [x,z]$, for every $z\in C$. It follows form Lemma~\ref{lem:int6}(2), that if exists, the gate is indeed unique. We say that $C$ is \textit{gate-convex} if for every $x\in M$, there exists a gate in $C$. In this case, we define the \textit{gate-projection} to $C$, $\pi_C:M\rightarrow C $, where $\pi_C(x)$ is the gate of $x$ in $C$. An example for gate-convex set is the interval $[x,y]$ with gate-projection $\pi_{[x,y]}(z)=m(x,y,z)$, for every $x,y\in M$. Notice that as a result of Lemma~\ref{lem:int6}(2), gate-convex sets are always convex.  

By \cite[Proposition~2.1]{fio20}, if $\phi: M\rightarrow M$ is a gate-projection, then for every $x,y,z\in M$ we have $\phi(m(x,y,z))=m(\phi(x),\phi(y),\phi(z))$. In particular, gate-projections map intervals to intervals. In addition, the following is true

\begin{lemma}[{\cite[Lemma~2.2]{fio20}}] \label{lem:fio2}
\begin{enumerate}
    \item If $C_1\subset M$ is convex and $C_2\subset M$ is gate-convex, the projection $\pi_{C_2}(C_1)$ is convex. If moreover, $C_1\cap C_2\neq \varnothing$, we have $\pi_{C_2}(C_1)=C_1\cap C_2$.
    \item If $C_1, C_2\subset M$ are gate-convex and $C_1\cap C_2\neq \varnothing $, then $C_1\cap C_2 $ is gate-convex with gate-projection $\pi_{C_2}\circ \pi_{C_1}=\pi_{C_1}\circ \pi_{C_2}$. In particular, if $C_2\subset C_1$, then $\pi_{C_2}=\pi_{C_2}\circ \pi_{C_1}$.
\end{enumerate}
\end{lemma}

Given a disjoint partition of a median algebra $M$ into non-empty convex subsets, $M=\mathfrak{h} \sqcup \mathfrak{h}^*$,
we say that $\mathfrak{h}$ and $\mathfrak{h}^*$ are complementary {\em half-spaces} in $M$ and we regard the unordered pair
$\mathfrak{w}=\{\mathfrak{h},\mathfrak{h}^*\}$ as a {\em wall} in $M$.
For disjoint subsets $A,B\subset M$, we say that $\mathfrak{h}$ {\em separates} $A$ from $B$ if
$A\subset \mathfrak{h}$ and $B\subset \mathfrak{h}^*$.
We denote by $\Delta(A,B)$ the collection of half-spaces that separate $A$ from $B$.
By \cite[Theorem 2.8]{rol98}, if $A$ and $B$ are convex then $\Delta(A,B)$ is not empty. 
Given a wall $\mathfrak{w}$,
we say that $\mathfrak{w}$ separates $A$ and $B$ if $\mathfrak{w}\cap \Delta(A,B)$ is not empty. 
For $a\notin B$ we denote $\Delta(a,B):=\Delta(\{a\},B)$.

\begin{lemma} \label{lem:pointrep}
For disjoint non-empty subsets $A,B\subset M$, if $A$ is gate-convex then there exists $a\in A$ 
such that $\Delta(A,B)=\Delta(a,B)$.
\end{lemma}

\begin{proof}
Fix $b\in B$ and let $a=\pi_A(b)$. Clearly, $\Delta(A,B) \subset \Delta(a,B)$. 
We will show the other inclusion.
Fix $\mathfrak{h}\in \Delta(a,B)$, and let $a' \in A\cap \mathfrak{h}^*$. By the definition of the gate-projection and the fact that $\mathfrak{h}^*$ is convex, we have $a\in [a',b]\subset\mathfrak{h}^*$, contradicting $a\in \mathfrak{h}$.
Thus, indeed, $A\subset \mathfrak{h}$ and we conclude that $\mathfrak{h}\in \Delta(A,B)$.
\end{proof}

We denote by $\mathscr{H}$ and by $\mathscr{W}$ the collections of all half-spaces and all walls in $M$, respectively.
There is a natural map $\mathscr{H}\to \mathscr{W}$, given by $\mathfrak{h}\mapsto \{\mathfrak{h}, \mathfrak{h}^*\}$.
Fix a section $\sigma:\mathscr{W}\to \mathscr{H}$, that is, a half-space representation for every wall in $M$. We denote by $\chi_A$ the characteristic function of a set A.
For each subset $W\subset \mathscr{W}$ we get a median morphism
\[ \iota_W:M\to \{0,1\}^W, \quad x\mapsto (\chi_{\sigma(\mathfrak{w})}(x))_\mathfrak{w}. \]
We say that a set of walls $W$ is {\em separating} if $\iota_W$ is injective,
that is, the walls in $W$ separate the points of $M$.
We say that $W$ is {\em transverse} if $\iota_W$ is surjective.
Note that these properties are independent of the choice of the section $\sigma$.
Two distinct walls $\mathfrak{w}_1,\mathfrak{w}_2\in \mathscr{W}$ are said to be transverse if $\{\mathfrak{w}_1,\mathfrak{w}_2\}$
is transverse.
We record the following lemma, which proof is an immediate application of Helly's Theorem, Lemma~\ref{lem:helly}.

\begin{lemma} \label{lem:trans}
If $W\subset \mathscr{W}$ is finite then
it is transverse if and only every pair of distinct walls in $W$ is transverse.
\end{lemma}

\section{Compact Median Algebras} \label{sec:compact}

A \textit{Topological median algebra} is a median algebra $M$ endowed with a  topology for which the median operator is continuous. 
In this note we will always assume that topologies are Hausdorff.

\begin{definition} \label{def:locc}
A topological median algebra $M$ is said to be \textit{locally convex} if each of its points has a basis of convex neighborhoods
and it is said to be \textit{locally open-convex} if each of its points has a basis of open and convex neighborhoods.
\end{definition}

\begin{example} \label{exm:locc}
\begin{enumerate}
\item The class of locally open-convex median algebras is closed under taking arbitrary products, hence for every index set I, 
the cube $\{0,1\}^I$ is a locally open-convex compact median algebra, with respect to the product topology and product median structure.
\item The class of locally open-convex median algebras is closed under taking median subalgebras. Therefore, any closed median subalgebra of a cube, such as the Roller compactification of a (possibly infinite dimensional) CAT(0) cube complex, is locally open-convex and compact.
\item By \cite[Lemma~3.1 and Lemma~3.2]{Bowditch2020}, every median space $(X,\rho)$ admits a pseudo-metric $\sigma$ which open balls are convex. Assume $X$ has a finite rank.
Then by \cite[Lemma~6.2]{Bowditch2020}, $\sigma$ is bilipschitz equivalent to $\rho$, 
thus it separates the points of $X$. It follows that $X$, as well as all its intervals, are locally open-convex. 
By \cite[Corollary 2.20]{fio20}, the completion of $X$ has compact intervals and its Roller compactification $\bar{X}$ is defined in 
\cite[Definition 4.13]{fio20}. By \cite[Definition 4.1]{fio20}, $\bar{X}$ is a median subalgebra of the product of all intervals in $X$.
As in the previous example, it follows that $\bar{X}$ is locally open-convex and compact.
\end{enumerate}
\end{example}

For the rest of the section we let $M$ be a {\em locally open-convex and compact topological median algebra}. 

We note that the closure of a convex set in $M$ is convex.
Indeed, if $C\subset M$ is convex, we have $C\times C\times M \subset m^{-1}(\bar{C})$
and, as $m^{-1}(\bar{C})$ is closed, we have $\bar{C}\times \bar{C}\times M\subset m^{-1}(\bar{C})$,
thus also $\bar{C}$ is convex.
Since $M$ is a normal topological space, it follows that every point in $M$ also has a basis of {\em closed} convex neighborhoods. 
We also note that, by \cite[Lemma 2.6 and Lemma 2.7]{fio20},
a convex set in $M$ is gate-convex if and only if it is closed and the gate-projections to the gate-convex sets are continuous.
In particular, every interval is closed (this is true in fact in every Hausdorff topological median algebra).

A half-space $\mathfrak{h}$ in $M$ is said to be {\em admissible} if it is open and $\mathfrak{h}^*$ has a non empty interior.
If further $\mathfrak{h}^*$ is open, we say that $\mathfrak{h}$ is {\em clopen}.
We denote by $\mathscr{H}^\circ$ the collections of all clopen half-spaces and 
we denote by $\mathscr{W}^\circ$ the collection of all corresponding walls.
For a subset $W\subset \mathscr{W}^\circ$, the median morphism $\iota_W:M\to \{0,1\}^W$ is clearly continuous.
By compactness, we get the following upgrade of Lemma~\ref{lem:trans}.

\begin{lemma} \label{lem:cube}
For every $W\subset \mathscr{W}^\circ$, $W$ is transverse
if and only if every pair of distinct walls in $W$ is transverse.
If moreover $W$ is separating then $\iota_W$ is an isomorphism of topological median algebras,
thus $M$ is a cube.
\end{lemma}

\begin{proof}
The first part follows immediately from Lemma~\ref{lem:trans} by the compactness of $M$.
If further $W$ is separating then $\iota_W$ is a continuous isomorphism of median algebras
and it is closed, as $M$ is compact, thus an homeomorphism.
\end{proof}

Since in a cube $\mathscr{W}^\circ$ is clearly transverse and separating, we immediately
get the following corollary of Lemma~\ref{lem:cube}.

\begin{corollary} \label{cor:cube}
$M$ is a cube if and only if $\mathscr{W}^\circ$ is separating and 
every pair of distinct walls in $\mathscr{W}^\circ$ is transverse.
\end{corollary}

We will use Corollary~\ref{cor:cube} in the next section.
In order to apply it we will use the existence of enough open and admissible half-spaces. 
These are provided by the next two results.

\begin{proposition} \label{prop:seperation1}
Let $U,C\subset M$ be two disjoint non-empty convex sets such that $U$ is open and $C$ is closed. 
Then there exists an open half-space $\mathfrak{h}\in \Delta(U,C)$.
\end{proposition}

\begin{proof}
By Lemma~\ref{lem:pointrep}, there exists a point $c\in C$ such that $\Delta(U,c)=\Delta(U,C)$.
We fix such $c$ and argue to show that there exists an open half-space $\mathfrak{h}\in \Delta(U,c)$.

We order the collection
\[ P= \{V\subset M \mid V \mbox{ is an open-convex set, } U\subset V \mbox{ and } c \notin V\} \]
by inclusion and note that it is non empty, as $U\in P$.
By Zorn's Lemma, $P$ has a maximal element,
as the union of any chain in $P$ forms an upper bound.
We fix such a maximal element $\mathfrak{h}$ and denote $\mathfrak{h}^*=M\setminus \mathfrak{h}$.
We argue to show that $\mathfrak{h}^*$ is convex, thus $\mathfrak{h}$ is a half-space.

Assume for the sake of contradiction that $\mathfrak{h}^*$ is not convex, that is, there exist 
$x,y\in \mathfrak{h}^*$ such that $[x,y]\cap \mathfrak{h}\neq \varnothing$. 
Fix $z\in [x,y]\cap \mathfrak{h}$ and set $\omega =\pi_{[x,z]}(c)=m(x,z,c)$.
In both cases, $\omega\in \mathfrak{h}$ and $\omega\in \mathfrak{h}^*$,
we will derive a contradiction to the maximality of $\mathfrak{h}$
by establishing $\mathfrak{h}\subsetneq V \in P$.

We assume first that $\omega\in \mathfrak{h}^*$ and denote $V\coloneqq \pi_{[\omega,z]}^{-1}(\mathfrak{h}\cap [\omega,z])$. Recalling that $\pi_{[\omega,z]}$ is a continuous median morphism, $V$ is open and convex.
Applying Lemma~\ref{lem:fio2}(1) for $C_1=\mathfrak{h}$ and $C_2=[\omega, z]$ we conclude that $\mathfrak{h}\subset V$.
In particular, $U\subset V$.
By Lemma~\ref{lem:int6}(1), since $\omega \in [z,x]$, we have $[z,\omega] \subset [z,x]$ and we get by Lemma~\ref{lem:fio2}(2),
$\pi_{[z,\omega]}(c)=\pi_{[z,\omega]} \circ \pi_{[z,x]}(c)=\pi_{[z,\omega]}(\omega)=\omega$.
Using our assumption that $\omega\in \mathfrak{h}^*$, it follows that $c\notin V$, thus $V\in P$.
Similarly, we use $z\in [x,y]$ 
to get
$\pi_{[z,\omega]}(y)=\pi_{[z,\omega]}\circ \pi_{[z,x]}(y)=\pi_{[z,\omega]}(z)=z\in  \mathfrak{h}\cap [\omega,z]$.
We conclude that $y\in V$, therefore $\mathfrak{h}\subsetneq V$,
contradicting the maximality of $\mathfrak{h}$.

We assume now that $\omega \in \mathfrak{h}$ and denote $V\coloneqq\pi_{[\omega,c]}^{-1}(\mathfrak{h}\cap [\omega,c])$.
Again, $V$ is clearly open and convex.
Applying now Lemma~\ref{lem:fio2}(1) for $C_1=\mathfrak{h}$ and $C_2=[\omega, c]$ we conclude that $\mathfrak{h}\subset V$.
Since $\pi_{[\omega,c]}(c)=c\notin \mathfrak{h}$ we have that $c\notin V$, thus $V\in P$.
Again, by Lemma~\ref{lem:int6}(1), as $\omega \in [z,c]$, we have $[\omega,c] \subset [z,c]$ and using Lemma~\ref{lem:fio2}(2) we get this time
$\pi_{[\omega,c]}(x)=\pi_{[\omega,c]}\circ \pi_{[z,c]}(x)=\pi_{[\omega,c]}(\omega)=\omega$,
which is in $\mathfrak{h}\cap [\omega,c]$ by our assumption that $\omega\in \mathfrak{h}$.
We conclude that $x\in V$, therefore $\mathfrak{h}\subsetneq V$,
contradicting again the maximality of $\mathfrak{h}$.
\end{proof}

\begin{proposition} \label{prop:seperation2}
Let $C,C' \subset M$ be two non-empty disjoint closed convex sets. Then there exists an admissible half-space $\mathfrak{h}\in \Delta (C,C')$.
\end{proposition}

\begin{proof}
Using Lemma~\ref{lem:pointrep} twice, we find $c\in C$ and $c\in C'$ such that $\Delta (c,c')=\Delta (C,C')$.
Using the fact that $M$ is normal, we find open-convex neighborhoods $c\in U$ and $c'\in V$ having disjoint closures.
By Proposition~\ref{prop:seperation1}, there exists an open half-space $\mathfrak{h}\in \Delta(U,\bar{V})$.
As $V\subset \mathfrak{h}^*$, $\mathfrak{h}$ is admissible.
We have $\mathfrak{h}\in \Delta (U,\bar{V})\subset \Delta (c,c')=\Delta (C,C')$,
thus indeed $\mathfrak{h}\in \Delta (C,C')$
\end{proof}

\section{The support of a balanced measure} \label{sec:support}

In this section we study the support of balanced measures.
Recall that the support of a measure is the minimal closed set having a null complement
and that every Borel measure on a compact second countable topological space has a non-empty support.

\begin{definition} \label{def:sclocc}
A topological median algebra $M$ is said to be \textit{sclocc} if it is
second countable, locally open-convex and compact.
\end{definition}

In view of Example~\ref{exm:locc}, we get the following.

\begin{example} \label{ex:sclocc}
\begin{enumerate}
\item For a countable index set I, the cube $\{0,1\}^I$ is sclocc.
\item The Roller compactification of a CAT(0) cube complex with countably many vertices is sclocc. 
\item By \cite[Theorem 4.14(4)]{fio20}, the Roller compactification of a second countable, finite rank median space is sclocc.
\end{enumerate}
\end{example}

For the rest of the section we let $M$ be a {\em sclocc median algebra}. 

We let $\Prob(M)$ be the space of probability Borel measures on $M$ and recall the definition of the self-median operator 
\[ \Phi:\Prob(M) \to \Prob(M), \quad \Phi(\mu)=m_* (\mu^3), \]
which fixed points are the balanced measures on $M$.
The following observation is trivial, but useful.

\begin{lemma} \label{lem:morphisms}
For any Borel median algebra morphism $M\to N$,
the push forward map $\Prob(M)\to \Prob(N)$ commutes with the corresponding self-median operators.
In particular, the image of a balanced measure on $M$ is a balanced measure on $N$.
\end{lemma}

Another basic lemma is the following.

\begin{lemma} \label{lem:01}
The balanced measures on $\{0,1\}$ are exactly $\delta_0$, $\delta_1$ and $\half\delta_0+\half\delta_1$.
\end{lemma}

\begin{proof}
Let $\mu$ be a balanced measure on $\{0,1\}$ and denote $x=\mu(\{0\})$. 
An easy calculation gives the equation $x=\mu(\{0\})=m_*(\mu^3)(\{0\})=x^3+3x^2(1-x)$
which solutions are exactly $0,1$ and $\half$.
\end{proof}

Since every Borel half-space gives a Borel median algebra morphism to $\{0,1\}$ (namely, the characteristic map of the half-space), we get the following.

\begin{corollary} \label{cor:half}
For every balanced measure on a median algebra,
the measure of any Borel half-space is either $0,1$ or $\half$.
\end{corollary}

For fully supported balanced measures and admissible half-spaces,
much more can be said.

\begin{lemma} \label{lem:balancedmeasures}
Let $\mu$ be a fully supported balanced measure on the sclocc median algebra $M$
and let $\mathfrak{h}$ be an admissible half-space. Then $\mathfrak{h}$ is clopen and $\mu(\mathfrak{h})=\half$.
If $\mathfrak{f}$ is a clopen half-space corresponding to another wall in $M$ then
$\mathfrak{f}$ and $\mathfrak{h}$ are transverse.
\end{lemma}

\begin{proof}
That $\mu(\mathfrak{h})=\half$ follows immediately from Corollary~\ref{cor:half}, by the assumptions that $\mu$ is fully supported and balanced, as by the admissibility of $\mathfrak{h}$,
both $\mathfrak{h}$ and $\mathfrak{h}^*$ have positive measures.
Fix $x\in \mathfrak{h}$ and use 
Proposition~\ref{prop:seperation2}
to find an admissible half-space $\mathfrak{h}' \in \Delta(\mathfrak{h}^*,x)$. 
Then also $\mu(\mathfrak{h}')=\half$ and we have $M=\mathfrak{h}\cup \mathfrak{h}'$.
It follows that $\mu(\mathfrak{h}\cap \mathfrak{h}')=0$. Since 
$\mathfrak{h}\cap \mathfrak{h}'$ is open and $\mu$ is fully supported, we conclude that $\mathfrak{h}\cap \mathfrak{h}'=\varnothing$,
thus $\mathfrak{h}'=\mathfrak{h}^*$ and indeed, $\mathfrak{h}$ is clopen

We now let $\mathfrak{f}$ be a clopen half-space which is not
transverse to $\mathfrak{h}$ and show that 
$\mathfrak{f}$ and $\mathfrak{h}$ determine the same wall in $M$.
Without loss of the generality we assume $\mathfrak{f}\cap\mathfrak{h}^*=\varnothing$.
We have $\mu(\mathfrak{f}\cap\mathfrak{h})=\mu(\mathfrak{f})=\half$ 
and as $\mu(\mathfrak{h})=\half$ we get $\mu(\mathfrak{f}^*\cap\mathfrak{h})=0$.
Since 
$\mathfrak{f}^*\cap \mathfrak{h}$ is open and $\mu$ is fully supported, we conclude that $\mathfrak{f}^*\cap\mathfrak{h}=\varnothing$,
thus $\mathfrak{f}=\mathfrak{h}$.
This finishes the proof .
\end{proof}

We are now ready to prove that the support of a balanced measure is a cube.

\begin{proof}[Proof of Proposition~\ref{prop:support}]
We let $\mu$ be a balanced measure on $M$.
We note that the support of $\mu$, $\supp(\mu)$, is a closed sub-algebra of $M$, and thus, a sclocc median algebra.
Indeed, fixing $x,y,z\in \supp(\mu)$, for every open neighborhood $O$ of  $m(x,y,z)$, there are open neighborhoods $x\in U$, $y\in V$ and $z\in W$ such that $(x,y,z)\in U\times V\times W \subset m^{-1}(O)$,
and we get $\mu(O)\geq \mu(U)\mu(V)\mu(W) >0$, therefore $m(x,y,z)\in \supp(\mu)$. By restricting to its support, we assume as we may that $\mu$ is a fully supported balanced measure on $M$ and argue to show that $M$ is a cube.

By Corollary~\ref{cor:cube}, we need to show that $\mathscr{W}^\circ$ is separating and 
every pair of distinct walls in $\mathscr{W}^\circ$ is transverse.
Proposition~\ref{prop:seperation2} guarantees that the collection of admissible half-spaces is separating
and, by the first part of Lemma~\ref{lem:balancedmeasures}, this collection coincides with $\mathscr{W}^\circ$.
Thus, we get that $\mathscr{W}^\circ$ is separating.
By the last part of Lemma~\ref{lem:balancedmeasures} we also get that every pair of distinct walls in $\mathscr{W}^\circ$ is transverse.
Thus, indeed, $M$ is a cube.
\end{proof}

\section{Balanced measures on cubes} \label{sec:uniform}

Fix a countable set $I$ and consider the cube $M=\{0,1\}^I$.
It is convenient to identify $\{0,1\}$ with the group $\mathbb{Z}/2\mathbb{Z}$ and $M$ with the compact group 
$(\mathbb{Z}/2\mathbb{Z})^I$.
We denote by $\lambda$ the Haar measure on $M$, $\lambda=(\half\delta_0+\half\delta_1)^I$.
It is fully supported and balanced.
To see that it is indeed balanced, recall that $\lambda$ is the unique probability Borel measure on $M$
that is invariant under translations, but by Lemma~\ref{lem:morphisms}, also $\Phi(\lambda)$ is invariant
under translations, as translations form median algebra automorphisms.

This section is devoted to the proof of Proposition~\ref{prop:uniform},
which claims that $\lambda$ is the unique fully supported balanced measure on $M$.
We observe that it is enough to prove this for finite cubes, that is, in case $|I|<\infty$,
as every cube is the inverse limit of its finite coordinate projections,
which are median algebras surjective morphisms,
thus take fully supported and balanced measure to fully supported and balanced measures
by Lemma~\ref{lem:morphisms}.

Dealing with measures on finite sets, we will identify a measure $\mu$ with the function $x\mapsto \mu(\{x\})$,
writing $\mu(x):=\mu(\{x\})$.
For our proof, it is beneficial to study a one-parameter class of measures on $M$, namely,
the measures that are invariant under translations by a certain index two subgroup $K_0<M$.

\begin{lemma} \label{lem:phidynamics}
Fix a natural integer $n$ and let $M=(\mathbb{Z}/2\mathbb{Z})^n$.
Consider the group homomorphism 
\[ \rho:M\to \mathbb{Z}/2\mathbb{Z}, \quad (x_1,\dots, x_n) \mapsto \sum x_i, \]
denote its kernel by $K_0$ and denote the non-trivial coset of $K_0$ by $K_1$.
Then the map 
\[ [0,1]\to \Prob(M)^{K_0}, \quad \omega\mapsto \mu_t=\frac{t}{2^{n-1}}\cdot\chi_{K_0}+\frac{1-t}{2^{n-1}}\cdot\chi_{K_1} \]
is bijective and the cubic polynomial 
\[ \phi(t)=t+(-1)^n2^{2-n}(t-\half)(t^2-t+\frac{1}{4}+(-1)^n\frac{3}{4}-(-1)^n2^{n-2}) \]
satisfies the relation, $\Phi(\mu_t)=\mu_{\phi(t)}$
for every $t\in [0,1]$.
\end{lemma}

We now provide the proof of Proposition~\ref{prop:uniform}, based on Lemma~\ref{lem:phidynamics}, 
which proof we postpone until later.

\begin{proof}[Proof of Proposition~\ref{prop:uniform}]
We argue to show that $\lambda$ is the unique fully supported balanced measure on $M=\{0,1\}^I$.
As mentioned above, we may assume that $I$ is finite. We do so
and prove the claim by an induction on $|I|$.
The base case $|I|=0$ is trivial. We note also that the case $|I|=1$ is proven in Lemma~\ref{lem:01}.
We now fix a natural $n>1$, assume $|I|=n$ and that the proposition is known for every cube $\{0,1\}^J$ with $|J|<n$.
We let $\mu$ be a balanced measure on $\{0,1\}^I$
and we argue to prove that for every $x\in \{0,1\}^I$, $\mu(x)=\lambda(x)=1/2^n$.

We denote by $\{e_i\mid i\in I\}$ the standard generating set of $M$
and for every $i\in I$ we consider the obvious projection $\pi_i:M\to \{0,1\}^{I\setminus \{i\}}$.
Fixing $i\in I$ and noticing that $\pi_i$ is a homomorphism of groups, we get by Lemma~\ref{lem:morphisms} that
the push forward of $\mu$ by $\pi_i$ is balanced, thus we conclude by our induction hypothesis
that for every $x\in M$, $\mu(x)+\mu(x+e_i)=1/2^{n-1}$.
It follows that for evey $i,j\in I$, $\mu(x)=\mu(x+e_i+e_j)$.
We denote by $K_0$ the group generated by the set $\{e_i+e_j\mid i,j\in I\}$
and conclude that $\mu$ is $K_0$ invariant.

Noticing that $K_0<M$ coincides with the subgroup considered in Lemma~\ref{lem:phidynamics},
it follows that $\mu=\mu_t$ for some $t\in [0,1]$.
Since $\mu$ is fully supported, we in fact  get that $t\in (0,1)$.
Since $\mu$ is balanced, we have by Lemma~\ref{lem:phidynamics},
$\mu_t=\Phi(\mu_t)=\mu_{\phi(t)}$,
thus $\phi(t)=t$.
We conclude that $t$ is a root of the polynomial 
\[ (-1)^n2^{n-2}(\phi(t)-t)=(t-\half)(t^2-t+\frac{1}{4}+(-1)^n\frac{3}{4}-(-1)^n2^{n-2}). \]
We observe that $t=1/2$ is the only root of this polynomial in the region $(0,1)$.
Indeed, for $c\in \mathbb{R}$, the polynomial $t^2-t+c$ has a root in $(0,1)$ iff $0<c \leq 1/4$,
which is not satisfied for $c=\frac{1}{4}+(-1)^n\frac{3}{4}-(-1)^n2^{n-2})$,
since for even $n$ we have $c\leq 0$ and for odd $n>1$ we have $c\geq 3/2$.
We conclude that $\mu=\mu_\half$,
thus indeed, for every $x\in \{0,1\}^I$, $\mu(x)=1/2^n$.
\end{proof}

\begin{proof}[Proof of Lemma~\ref{lem:phidynamics}]
The map $t\mapsto \mu_t$ is clearly injective and it is onto $\Prob(M)^{K_0}$
as every measure in $\Prob(M)^{K_0}$ (considered as a function on $M$) is constant on the fibers of $\rho$.
Since translation on $M$ are median algebra automorphisms, we get by Lemma~\ref{lem:morphisms} that
$\Phi$ preserves $\Prob(M)^{K_0}$, thus for every $t\in [0,1]$, $\Phi(\mu_t)=\mu_{\psi(t)}$
for some function $\psi:[0,1]\to [0,1]$.
We are left to prove that $\psi=\phi$. 

We denote by $0\in M$ the group identity and set $X=m^{-1}(\{0\})\subset M^3$.
Clearly, $0\in K_0$, thus for every $t\in [0,1]$, $\mu_t(0)=t/2^{n-1}$.
Applying this to $\psi(t)$, we get that for every $t\in [0,1]$, 
\[ \mu_t^3(X)=\mu_t^3(m^{-1}(\{0\}))=m_*(\mu_t^3)(\{0\})=\Phi(\mu_t)(\{0\})=\mu_{\psi(t)}(0)=\psi(t)/2^{n-1}. \]
It is then enough to show that for every $t\in [0,1]$ we have $\mu_t^3(X)=\phi(t)/2^{n-1}$,
which is what we now proceed to show.

We need to understand the subset $X\subset M^3$ and its measure under $\mu_t^3$.
We use the decomposition $M^3=A_0 \sqcup A_1 \sqcup A_2 \sqcup A_3$,
where 
\[ A_i=\cup \{K_{\epsilon_1}\times K_{\epsilon_2}\times K_{\epsilon_3} \mid \epsilon_1,\epsilon_2,\epsilon_3 \in\{0,1\},
~\epsilon_1+\epsilon_2+\epsilon_3=i\}, \]
that is, $A_i$ is the subset of $M^3$ consists of triples of elements out of which exactly $i$ are in $K_1$.
We observe that $\mu_t^3$, as a function on $M^3$, attains the constant value $t^{3-i}(1-t)^i/8^{n-1}$ on $A_i$.
Denoting $X_i=X\cap A_i$ and $a_i=|X_i|$ we get the formula 
\begin{equation} \label{eq:mu^3}
\mu_t^3(X)=\sum_{i=0}^3 a_i\cdot \frac{t^{3-i}(1-t)^i}{8^{n-1}}. 
\end{equation}

Our next goal will be to compute the coefficients $a_i$.
For this we now emphasize their dependence on $n$, denoting them $a_i(n)$.
Similarly, we write $X(n)$ and $X_i(n)$ for $X$ and $X_i$ correspondingly.
Writing further, $M(n)=\{0,1\}^n$, we make the identification $M(n)\simeq M(n-1)\times \{0,1\}$
and we identify accordingly also $M(n)^3\simeq M(n-1)^3 \times \{0,1\}^3$.
As $0\mapsto 0$ under the projection map $M(n)\to M(n-1)$,
we clearly have, using Lemma~\ref{lem:morphisms}, that the image of $X(n)$ under the corresponding map $M(n)^3\to M(n-1)^3$
is in $X(n-1)$.
We denote by $\pi:X(n) \to X(n-1)$ the corresponding restriction map
and consider the partition $X(n-1)=\sqcup_{j=0}^3 X_j(n-1)$.
We fix $i,j\in \{0,1,2,3\}$ and for each $x\in X_j(n-1)$ count the intersection size of the fiber $\pi^{-1}(\{x\})$
with the sets $X_i(n)$, that is, $|\pi^{-1}(\{x\})\cap X_i(n)|$.
One verifies easily that this size does not depends on $x$, only on $i,j\in \{0,1,2,3\}$
and, denoting it by $s_{i,j}$, we have 
\[ (s_{i,j})=\begin{pmatrix} 
1 & 1 & 0 & 0 \\
3 & 1 & 2 & 0 \\
0 & 2 & 1 & 3 \\
0 & 0 & 1 & 1
\end{pmatrix} . \]
We therefore obtain the recurrence linear relation
\[
\begin{pmatrix}
a_0(n) \\ a_1(n) \\ a_2(n) \\ a_3(n)
\end{pmatrix} 
=
\begin{pmatrix} 
1 & 1 & 0 & 0 \\
3 & 1 & 2 & 0 \\
0 & 2 & 1 & 3 \\
0 & 0 & 1 & 1
\end{pmatrix}
\begin{pmatrix}
a_0(n-1) \\ a_1(n-1) \\ a_2(n-1) \\ a_3(n-1)
\end{pmatrix},
\quad \quad
\begin{pmatrix}
a_0(1) \\ a_1(1) \\ a_2(1) \\ a_3(1)
\end{pmatrix} 
=
\begin{pmatrix}
1 \\ 3 \\ 0 \\ 0
\end{pmatrix},
\]
that leads to the explicit formulas
 \begin{align*}
     a_0(n) & = 2^n\left((\frac{3}{8}+(-1)^n \frac{1}{8})+2^n\frac{1}{8}\right), \\ 
     a_1(n) & = 2^n\left((\frac{3}{8}-(-1)^n \frac{3}{8})+2^n \frac{3}{8}\right), \\
     a_2(n) & = 2^n\left((-\frac{3}{8}+(-1)^n \frac{3}{8})+2^n \frac{3}{8}\right), \\
     a_3(n) & = 2^n\left((-\frac{3}{8}-(-1)^n\frac{1}{8})+2^n \frac{1}{8}\right).
 \end{align*}
By substituting these values in equation~\eqref{eq:mu^3}, we get
\begin{align*}
\mu_t^3(X)  & = \sum_{i=0}^3 a_i\cdot \frac{t^{3-i}(1-t)^i}{8^{n-1}} \\
 & = 2^n\left((\frac{3}{8}+(-1)^n \frac{1}{8})+2^n\frac{1}{8}\right)\cdot\frac{t^{3}}{8^{n-1}} \\
 & + 2^n\left((\frac{3}{8}-(-1)^n \frac{3}{8})+2^n \frac{3}{8}\right)\cdot\frac{t^{2}(1-t)}{8^{n-1}} \\
 & + 2^n\left((-\frac{3}{8}+(-1)^n \frac{3}{8})+2^n \frac{3}{8}\right)\cdot\frac{t(1-t)^2}{8^{n-1}} \\
 & + 2^n\left((-\frac{3}{8}-(-1)^n\frac{1}{8})+2^n \frac{1}{8}\right)\cdot\frac{(1-t)^3}{8^{n-1}} \\
 & = \frac{t+(-1)^n2^{2-n}(t-\half)(t^2-t+\frac{1}{4}+(-1)^n\frac{3}{4}-(-1)^n2^{n-2})}{2^{n-1}}= \frac{\phi(t)}{2^{n-1}},
\end{align*}
thus indeed, $\mu_t^3(X)=\phi(t)/2^{n-1}$, and this finishes the proof.
\end{proof}

\end{document}